\documentclass[10pt,a4paper]{amsart}
\usepackage[utf8]{inputenc}
\usepackage[english]{babel}
\usepackage{amsmath}
\usepackage{amsfonts}
\usepackage{amssymb}
\usepackage{amsaddr}
\usepackage{comment}

\usepackage[hidelinks]{hyperref}

\usepackage{tikz-cd}

\numberwithin{equation}{section}

\theoremstyle{plain}
\newtheorem{theorem}{Theorem}[section]
\newtheorem{lemma}[theorem]{Lemma}

\newtheorem{prop}[theorem]{Proposition}
\newtheorem{coro}[theorem]{Corollary}

\newtheorem{theoremA}{Theorem}

\theoremstyle{definition}

\newtheorem{example}[theorem]{Example}
\newtheorem{rema}[theorem]{Remark}

\newcommand{\bR}{\mathbb{R}}

\newcommand{\bQ}{\mathbb{Q}}
\newcommand{\bN}{\mathbb{N}}

\newcommand{\bP}{\mathbb{P}}

\newcommand{\cO}{\mathcal{O}}

\DeclareMathOperator{\Bs}{\mathrm{Bs}}

\DeclareMathOperator{\Hc}{\mathrm{\mathrm{H}^0}}
\DeclareMathOperator{\codim}{\mathrm{codim}}
\DeclareMathOperator{\vol}{\mathrm{vol}}

\DeclareMathOperator{\ord}{\mathrm{ord}}
\DeclareMathOperator{\mult}{\mathrm{mult}}

\author[Fran\c{c}ois Balla\"{y}]{\small{Fran\c{c}ois Balla\"{y}}}
\title[Lower bounds for Seshadri constants via successive minima]{Lower bounds for Seshadri constants via successive minima of line bundles}
\date{\today}

\begin{document}

\begin{abstract} Given  a nef and big line bundle $L$ on a projective variety $X$ of dimension $d \geq 2$, we prove that the Seshadri constant of $L$ at a very general point is larger than $(d+1)^{\frac{1}{d}-1}$. This slightly improves the lower bound $1/d$ established by Ein, Küchle and Lazarsfeld. The proof relies on the concept of successive minima for line bundles recently introduced by Ambro and Ito.
\end{abstract}

\maketitle

 Let $X$ be a projective variety of dimension $d\geq 2$ over an uncountable algebraically closed field $k$ of characteristic zero. Given a nef and big line bundle $L$ on $X$, the Seshadri constant of $L$ at a closed point $x \in X(k)$ is a numerical invariant introduced by Demailly \cite{Dem92} measuring the local positivity of $L$ at $x$. It is defined by
\[\epsilon(L,x) = \inf_{C \ni x} \frac{L\cdot C}{\mult_xC},\]
where the infimum is over all integral curves in $X$ passing through $x$. Seshadri constants  have received a lot of attention since their introduction, and have become an important tool to study the geometry of projective varieties. We refer the reader to \cite[Chapter 5]{Lazarsfeld} for motivation and background on this invariant.  Lower bounds for Seshadri constants are particularly important, notably because they provide effective results for the existence of global sections of adjoint bundles and generic statements towards Fujita's freeness conjecture  (see \cite[Proposition 6.8]{Dem92}, \cite[Theorem A]{Deng}). 

Examples of Miranda show that one cannot expect an absolute lower bound for $\epsilon(L,x)$ to hold at every point: for every $\delta > 0$, there exist $X$, $L$ and $x$ as above such that $\epsilon(L,x) < \delta$. However, a celebrated conjecture of Ein and Lazarsfeld predicts that the lower bound $1$ holds at very general points, namely that there exists a countable union $ \mathcal{B} \varsubsetneq X$ of proper subvarieties such that $\epsilon(L,x) \geq 1$ for every closed point $x\in X \setminus \mathcal{B}$ (\cite[Problem 3.3]{EL}, \cite[Conjecture 5.2.4]{Lazarsfeld}). When $X$ is a surface, this conjecture is a theorem of Ein and Lazarsfeld  
 \cite{EL}. In arbitrary dimension, Ein, Küchle and Lazarsfeld \cite{EKL} obtained the lower bound $1/d$. The approach of \cite{EKL} was subsequently refined by Nakamaye \cite{Na04}, who proved that 
 \begin{equation}\label{eqNak}
 \epsilon(L,x) > \frac{3d +1}{3d^2}
 \end{equation}
for a very general point $x \in X(k)$. To the best of our knowledge, this is the sharpest lower bound known up to now at this level of generality. In dimension $3$, Cascini and Nakamaye \cite{CasciniNakamaye} obtained the lower bound $ \epsilon(L,x) > 1/2$. Our main result improves the lower bound \eqref{eqNak} as follows (see section \ref{sectionproofmain} for more precise statements).
\begin{theoremA}\label{thmmainintro} For a very general point $x \in X(k)$, we have 
\[\epsilon(L,x) \geq  (d+1)^{\frac{1}{d}-1}.\]
\end{theoremA}
 In particular, this theorem shows that $ \epsilon(L,x) > 1/(d-\frac{1}{2})$ for every $d \geq 4$, and that $ \epsilon(L,x) > 1/(d-1)$ if $d \geq 7$. It can therefore be considered as a partial generalization of the main result of \cite{CasciniNakamaye} to higher dimensions.
 
  Our proof of Theorem \ref{thmmainintro} is strongly inspired by the methods of \cite{EKL} and \cite{Na04}. In particular, it relies crucially on a powerful differentiation result due to Ein, Küchle and Lazarsfeld \cite{EKL}. The novelty in our approach is to combine the latter with the notion of successive minima for line bundles recently introduced by Ambro and Ito \cite{AmbroIto}. This significantly simplifies  the ``gap argument" and  the construction of a flag of auxiliary subvarieties used in \cite{EKL, Na04}. In order to outline our strategy, we fix a very general point $x \in X(k)$. The successive minima of $L$ at $x$ form a chain of real numbers 
\[0 = \epsilon_{d+1}(L,x)\leq \epsilon_d(L,x) \leq \epsilon_{d-1}(L,x) \leq \cdots \leq \epsilon_2(L,x) \leq \epsilon_1(L,x),\]
measuring the local positivity of $L$ at $x$, with $\epsilon_d(L,x)  = \epsilon(L,x) $ and $\epsilon_1(L,x) \geq \sqrt[d]{L^d} \geq 1$. More precisely, for a real number $t$ we let $\Bs |\mathcal{I}_x^{t+}L|_\bQ$ be the locus where all the global sections $s$ of $mL$ with $\ord_x s > mt$ vanish, $m \in \bN$. The numbers $\epsilon_i(L,x)$ detect the jumps in the codimension at $x$ of $\Bs |\mathcal{I}_x^{t+}L|_\bQ$ when $t\geq 0$ varies. Assume that there exists $i \in \{1, \ldots, d-1\}$ such that $\epsilon_i(L,x) > \epsilon_{i+1}(L,x)$, and let $t,t'$ be rational numbers with  
\[\epsilon_{i+1}(L,x) < t < t' < \epsilon_{i}(L,x).\]
It follows from the definition of the successive minima that the base loci  $\Bs |\mathcal{I}_x^{t+}L|_\bQ \subseteq \Bs |\mathcal{I}_x^{t'+}L|_\bQ$  share an irreducible component $Z$ containing $x$. This implies that for a sufficiently general point $z \in Z(k)$, we have 
 ${L \cdot C}\geq \epsilon(L_{|Z}, z){\mult_z C}$
for every curve $C \ni z$ contained in $\Bs |\mathcal{I}_x^{t'+}L|_\bQ$. On the other hand, the differentiation result of \cite{EKL} (see Lemma \ref{lemmadiff}) implies that 
 $L \cdot C\geq (t'-t)\mult_z C $
 for any curve $C \ni z$ with $ C \nsubseteq \Bs |\mathcal{I}_x^{t'+}L|_\bQ$. Using a classical semi-continuity property for Seshadri constants (Lemma \ref{lemmaverygeneral}), we obtain 
 \[\epsilon(L,x) \geq \epsilon(L,z) = \inf_{C \ni z } \frac{L \cdot C}{\mult_z C}  \geq \min\{t'-t, \epsilon(L_{|Z}, z) \}.\]
 Arguing by induction on the dimension, all that remains to prove Theorem \ref{thmmainintro} is to control the gaps $\epsilon_i(L,x) - \epsilon_{i+1}(L,x)$ between successive minima. To do so, we use a variant of Minkowski's second theorem due to Ambro and Ito, from which we derive the inequality 
 \[\max_{1 \leq i \leq d} \epsilon_i(L,x) - \epsilon_{i+1}(L,x) \geq \frac{\sqrt[d]{(d+1)L^d}}{d+1} \geq  (d+1)^{\frac{1}{d}-1}\]
 (see Lemma \ref{ineqmin}). 
 
 The above strategy actually leads to more precise versions of Theorem \ref{thmmainintro}, that we shall state and prove in section \ref{sectionproofmain}. In section \ref{sectionprelim} we recall preliminary results, including the differentiation result from \cite{EKL} as well as the definition and important facts on successive minima for line bundles from \cite{AmbroIto}.

\subsection*{Notation and conventions}\label{sectiondef}  We fix an uncountable algebraically closed field $k$ of characteristic zero, and all schemes are defined over $k$.
 A projective variety $X$ is an integral projective scheme. 
  We say that a property is true at a very general point of  $X$ if it holds for all $x \in X(k)$ outside a given countable union of proper subvarieties in $X$. 
   The volume of a line bundle $L$ on $X$ is the quantity 
\[\vol(L) := \limsup_{m \rightarrow \infty} \frac{h^0(X,mL)}{m^{\dim X}/(\dim X)!}.\]
 We say that $L$ is big if $\vol(L) > 0$. If $L$ is big, there exists a open subset $U$ of $X$ such that $L_{|Z}$ is big for every subvariety $Z$ with $Z \cap U \ne \emptyset$.  When $L$ is nef, we have $\vol(L)=L^d$ by the asymptotic Riemann-Roch theorem.

\section{Seshadri constants, order of global sections and successive minima}\label{sectionprelim}
 This section contains the preliminary material that we need for the proof of Theorem \ref{thmmainintro}. We recall a well-known semi-continuity property for Seshadri constants in subsection \ref{paragSesh}, and a consequence of the differentiation lemma from \cite{EKL} in subsection \ref{paragorder}. Subsections \ref{paragmin} and \ref{paragmingeneral} are devoted to successive minima for line bundles, following Ambro and Ito \cite{AmbroIto}. Throughout this section, $X$ denotes a projective variety of dimension $d \geq 1$ and $L$ is a line bundle on $X$. 
\subsection{Seshadri constants at very general points}\label{paragSesh} The following lemma is well-known, and it is essentially  \cite[Lemma 1.4]{EKL}.
\begin{lemma}\label{lemmaverygeneral} Assume that $L$ is  nef and big. For any smooth point $y \in X(k)$ and any $\delta > 0$, there exists a dense open subset $U\subseteq X$ with 
\[U(k) \subseteq \{x \in X(k) \ | \ \epsilon(L,x) >  \epsilon(L,y)- \delta\}.\]
\end{lemma}

\begin{proof} In the case where $L$ is ample, the result is \cite[Proposition 2.5.12]{DFEM}. The general case follows as in \cite[Proof of Lemma 1.4]{EKL}.  

\end{proof}

 If $X^{\mathrm{sm}}$ denotes the smooth locus of $X$, it follows from Lemma \ref{lemmaverygeneral} that the supremum 
\[\epsilon(L,1) = \sup_{x \in X^{\mathrm{sm}}(k)} \epsilon(L,x)\]
is actually a maximum, and that $\epsilon(L,1) = \epsilon(L,x)$ for a very general $x \in X(k)$. In particular, we have 
$\epsilon(L,1) =\max_{x \in X^{\mathrm{sm}}(k)} \epsilon(L,x) \geq \epsilon(L,y)$ for any $y \in X^{\mathrm{sm}}(k)$. 
 
\subsection{Order of global sections}\label{paragorder}

 Let $s \in \Hc(X,L)$ be a global section. We denote by $\ord_x s$ the order of $s$ at a closed point $x \in X(k)$, defined to be supremum of the integers $m \geq 0$ such that $s \in \mathcal{I}_x^mL$, where $\mathcal{I}_x$ is the ideal defining $x \in X$. By definition $\ord_x s = +\infty$ if $s =0$ and $\ord_x s \in \bN$ otherwise. We denote by $Z(s) \subseteq X$ the closed subset on which $s$ vanishes. 

Let $t\geq 0$ be a real number and let $m \geq 0$ be an integer. We denote by $R_m^t(L)$ the linear subspace of $R_m(L) := \Hc(X,mL)$ consisting of the global sections $s \in R_m(L)$ with $\ord_x s > mt$. Following Ambro and Ito \cite[Section 2]{AmbroIto}, we define a closed subset of $X$ by
\[\Bs |\mathcal{I}_x^{t+}L|_\bQ = \bigcap_{m \in \bN} \bigcap_{s \in R_m^t(L)} Z(s).\]
The following consequence of \cite[Proposition 2.3]{EKL} is a key ingredient in our proof of Theorem \ref{thmmainintro}.
\begin{lemma}[\cite{AmbroIto}, Lemma 2.18]\label{lemmadiff} 
 Let $t \geq 0$ be a real number. Let $x \in X(k)$ be a very general point and let $Z \subseteq  \Bs |\mathcal{I}_x^{t+}L|_\bQ$ be an irreducible component containing $x$. Then for any integers $p,q \geq 1$ and any $s \in \Hc(X,\mathcal{I}_x^{p}(qL))$, we have $\ord_z s \geq p-qt$ for every $z \in Z(k)$ in the smooth locus of $X$. 

\end{lemma}

\begin{proof} By \cite[Lemma 2.18]{AmbroIto}, there exists a dense open subset $U \subseteq Z$ such that $\ord_z s \geq p-qt$ for every $z \in U(k)$. On the other hand, the function 
$x \mapsto \ord_x s$ is upper semi-continuous 
 on the smooth locus $X^{\mathrm{sm}}$ of $X$ (see \cite[section 2]{EKL}). Therefore the Zariski-closure $\overline{U}$ of $U \cap X^{\mathrm{sm}}$ in $X^{\mathrm{sm}}$ satisfies 
\[\overline{U}(k) \subseteq \{ x \in X^{\mathrm{sm}}(k) \ | \ \ord_x s \geq p-qt\},\]
and the result follows since $\overline{U} = Z \cap X^{\mathrm{sm}}$ by construction.
\end{proof}

\subsection{Successive minima for line bundles}\label{paragmin}  
 Let $x\in X(k)$ be a closed point.
  For any positive integer $i$, Ambro and Ito \cite{AmbroIto} introduced   the \emph{$i$-th successive minimum of $L$ at $x$} defined by
\[\epsilon_i(L,x) = \inf \{t \geq 0 \ | \ \codim_x \Bs |\mathcal{I}_x^{t+}L|_\bQ < i\}.\]
These real numbers form a chain 
\[\epsilon_1(L,x) \geq \epsilon_2(L,x) \geq \cdots \geq \epsilon_d(L,x) \geq \epsilon_{d+1}(L,x)=0.\]
By \cite[Corollary 3.3]{AmbroIto},  the extremal minima can be compared to the volume of $L$ as follows.
\[\epsilon_d(L,x) \leq \sqrt[d]{\frac{\vol(L)}{\mult_x X}} \leq \epsilon_1(L,x).\]
 It is a non-trivial fact that the last minimum $\epsilon_d(L,x)$ coincides with the Seshadri constant of $L$ at $x$ when $L$ is nef.
 \begin{prop}[\cite{AmbroIto}, Proposition 2.20]\label{propSeshmin} If $L$ is nef, then $\epsilon_d(L,x) = \epsilon(L,x)$. 
\end{prop}
 The following remark gives  an alternative description of the first minimum.
 \begin{rema}\label{remarkwidth} If $L$ is big, then 
 \[\epsilon_1(L,x) = \sup \left\{\frac{\ord_x s}{m} \ | \ s \in \Hc(X,mL) \setminus \{0\}, \  m\geq 1 \right\}\]
 for any $x \in X(k)$.
 If moreover $x$ is a smooth point, then 
 \[\epsilon_1(L,x) = \sup \{ t \in \bQ_{\geq 0} \ | \ \pi_x^*L - tE_x \ \text{ is big } \},\]
 where $\pi_x \colon \mathrm{Bl}_x X \rightarrow X$ is the blow-up at $x$ and $E_x$ is the exceptional divisor (see \cite[page 14]{AmbroIto}). 
 \end{rema}
 
   We now recall two examples from \cite[Remark 3.8]{AmbroIto}.
\begin{example}If $L = \cO_{\bP^d}(a)$ for some integer $a > 0$, then $\epsilon_i(L,x) = a$ for any $x \in \bP^d(k)$ and $i \in \{1, \ldots,d\}$.  
 Let us now consider  the product of $d$ copies of the projective line $X = (\bP^1_k)^d$, $d \geq 1$. Let $L = \cO_X(\omega_1, \ldots, \omega_d)$, where $\omega_1 \geq \cdots \geq \omega_d$ are positive integers. Then $\epsilon_i(L,x) = \sum_{j=i}^d \omega_i$ for any $x \in X(k)$ and $i \in \{1, \ldots,d\}$. In particular, $\epsilon(L,x) = \epsilon_d(L,x) = \omega_d$ and $\epsilon_1(L,x) = \omega_1 + \cdots +\omega_d$ for any $x \in X(k)$.
\end{example}

\subsection{Successive minima at very general points}\label{paragmingeneral}

Successive minima don't satisfy a semi-continuity property analogous to Lemma \ref{lemmaverygeneral} in general. However, there exists a countable union $\mathcal{B}_L = \cup_{n\in \bN} Y_n \varsubsetneq X$ of proper subvarieties such that for any integer $i > 0$, the function $x \mapsto \epsilon_i(L,x)$ is constant on $X(k)\setminus \mathcal{B}_L$  (see \cite[Proposition 2.12]{AmbroIto}). We let $\epsilon_i(L) = \epsilon_i(L,x)$, where $x \notin \mathcal{B}_L$ is a closed point. In particular, $\epsilon_d(L) = \epsilon(L,1)$ is the Seshadri constant of $L$ at a very general point when $L$ is nef (see Proposition \ref{propSeshmin}). 

The following is an analogue  for successive minima of line bundles of Minkowski's second theorem in geometry of numbers.
 \begin{theorem}[\cite{AmbroIto}, Theorem 3.6]\label{thmMinkowski} If $L$ is big, then we have 
 \[\prod_{i=1}^d \epsilon_i(L) \leq \vol(L) \leq d!\prod_{i=1}^d \epsilon_i(L).\]
 \end{theorem}
 For real numbers $t_1, \ldots, t_d$, we consider the compact convex set $\square(t_1, \ldots, t_d) \subset \bR^d$ defined by
 \[\square(t_1, \ldots, t_d) = \cap_{i=1}^d\{(x_1, \ldots, x_d) \in \bR^d_{\geq 0} \ | \ x_i + \ldots + x_d \leq t_i\}.\]
 With this notation, the upper bound of Theorem \ref{thmMinkowski} is a consequence of the following stronger result due to Ambro and Ito.
 \begin{prop}[\cite{AmbroIto}, Proposition 3.5]\label{propineqminvol} 
  We have
 \[\vol(L) \leq d! \vol(\square(\epsilon_1(L), \ldots, \epsilon_d(L))).\]
 
 \end{prop}
 Proposition \ref{propineqminvol} gives control on the  gaps between successive minima as follows. 
\begin{lemma}\label{ineqmin}  We have 
\[\max_{1 \leq i \leq d} \epsilon_i(L) - \epsilon_{i+1}(L) \geq  \frac{\sqrt[d]{(d+1)\vol(L)}}{d+1}.\] 
\end{lemma}

\begin{proof} For two real numbers $\alpha, \beta \geq 0$, let 
\begin{equation*}
\begin{split}
\square_d(\alpha, \beta)& = \square(\alpha+ (d-1)\beta), \alpha + (d-2)\beta, \ldots, \alpha +\beta, \alpha)\\
& = \{(x_1, \ldots, x_d) \in \bR_{\geq 0}^d \ | \   \forall\  1 \leq j \leq d, \ \sum_{i = j}^d x_i \leq \alpha + (d-j)\beta  \}.
\end{split}
\end{equation*}
Using the identity 
\begin{equation*}
\vol(\square_d(\alpha,\beta)) = \int_0^\alpha \vol(\square_{d-1}(\alpha+\beta-t,\beta))dt,
\end{equation*}
one can show by induction that $d!\vol(\square_d(\alpha, \beta)) = \alpha(\alpha + d\beta)^{d-1}$. Let $\theta(L) = \max_{1 \leq i \leq d-1} \epsilon_i(L) - \epsilon_{i+1}(L)$. 
By Proposition \ref{propineqminvol} and by definition of $\theta(L)$, it follows that
\[\vol(L) \leq d!\vol(\square_d(\epsilon_d(L), \theta(L)) = (\epsilon_d(L) + d\theta(L))^{d-1}\epsilon_d(L),\]
and therefore 
\[\vol(L) \leq (d+1)^{d-1} (\max_{1 \leq i \leq d} \epsilon_i(L) - \epsilon_{i+1}(L))^d.\]
\end{proof}

\section{Main theorem}\label{sectionproofmain}

 We shall deduce Theorem \ref{thmmainintro} from the following theorem, which is the central result of this paper. 

\begin{theorem}\label{thmmain} Let $X$ be a projective variety of dimension $d \geq 2$ and let $L$ be a nef and big line bundle on $X$. Let $\mathcal{B} \varsubsetneq X$ be a  countable union of proper subvarieties and let $\alpha$ be a  real number with
\[\alpha < \max_{1 \leq i \leq d-1} \epsilon_i(L) - \epsilon_{i+1}(L).\]
Then there exists a proper subvariety $Z_{\alpha} \varsubsetneq X$ with $Z_\alpha \nsubseteq \mathcal{B}$ and $0 < \dim Z_\alpha < d$, and there exists a dense open subset $U_\alpha \subseteq Z_\alpha$ such that
\[\epsilon(L,z) \geq \min \left\{ \alpha, \epsilon(L_{|Z_{\alpha}},z) \right\}\]
for any point $z \in U_\alpha(k)$.
\end{theorem}

\begin{proof} Let $\alpha$ be a real number  with
\[ \alpha <  \max_{1 \leq i \leq d-1} \epsilon_i(L) - \epsilon_{i+1}(L),\]
and let  $\mathcal{B} \nsubseteq X$ be a  countable union of subvarieties. We assume that $\alpha > 0$, otherwise there is nothing to prove. For a very general point $x \in X(k)\setminus \mathcal{B}$ we have $\epsilon_i(L,x) = \epsilon_i(L)$ for any integer $i > 0$. There exists an index $ i \in \{1, \ldots, d-1\}$ such that 
 \[\epsilon_i(L,x) - \epsilon_{i+1}(L,x) =  \max_{1 \leq j \leq d-1} \epsilon_j(L) - \epsilon_{j+1}(L) > \alpha,\]
 and therefore there are rational numbers $t, t'$ such that $t' - t > \alpha$ and
 \[ \epsilon_{i+1}(L,x) < t < t' < \epsilon_i(L,x). \]
 Let $V = \Bs |\mathcal{I}_x^{t+}L|_\bQ$ and $V' = \Bs |\mathcal{I}_x^{t'+}L|_\bQ$. By definition of the successive minima of $L$ at $x$, we have 
 \[ \codim_x V = \codim_x V' = i,\]
 hence $V$ and $V'$ have a common irreducible component $Z \ni x$ with $\codim_X Z = \codim_x V = i$. In particular $Z \nsubseteq \mathcal{B}$ and $1 \leq \dim Z \leq d-1$. Let $z \in Z(k)$ be a smooth point of $Z$ and $X$ such that $Z$ is the only irreducible component of $V'$ containing $z$.  Let $C \subseteq X$ be an integral curve containing $z$. If $C \subseteq V'$, then $C \subseteq Z$. Since $Z$ is smooth at $z \in C$, $Z$ is smooth at a general point of $C$  and it follows that  
 \[\frac{L \cdot C}{\mult_z C}=\frac{L_{|Z} \cdot C}{\mult_z C} \geq \epsilon(L_{|Z},z).\] 
 Assume now that $C \nsubseteq V' = \Bs |\mathcal{I}_x^{t'+}L|_\bQ$. There exist an integer $m \geq 1$ such that $mt' \in \bN$ and a non-zero  global section $s \in \Hc(X,I_x^{mt'}(mL))$ such that $C \nsubseteq Z(s)$. Since $Z$ is an irreducible component of $V =  \Bs |\mathcal{I}_x^{t+}L|_\bQ$ and $z \in Z$ is in the smooth locus of $X$, Lemma \ref{lemmadiff} implies that $\ord_z s \geq m(t'-t)$. Since $C$ and $Z(s)$ intersect properly, we have
 \[\frac{L \cdot C}{\mult_z C} = \frac{1}{m} \frac{mL \cdot C}{\mult_z C} \geq \frac{1}{m} \frac{\ord_z(s) \mult_z(C)}{\mult_z(C)} \geq t'-t > \alpha.\]
 We have proved that for any integral curve $C \subseteq X$ containing $z$, 
 \[\frac{L \cdot C}{\mult_z C} \geq \min \{ \alpha,  \epsilon(L_{|Z}, z) \}.\]
 Therefore we have
 \begin{equation*}
  \epsilon(L,z)  \geq \min \{ \alpha,  \epsilon(L_{|Z}, z) \}.
 \end{equation*}
Since the point $z$ can be chosen arbitrarily in an open subset of $Z$, the theorem is proved.
\end{proof}

  Given a line bundle $L$ on a projective variety of dimension $d\geq 1$, we let
  \[ \gamma(L) = \max\left\{\frac{\sqrt[d]{(d+1)L^d}}{d+1}, \frac{\epsilon_1(L)}{d} \right\}.\]
  Applying  Lemma \ref{ineqmin} and Theorem \ref{thmmain} inductively, we have the following corollary. 
 \begin{coro}\label{coromain} Let $X$ be a projective variety of dimension $d \geq 1$, let $L$ be a nef and big line bundle on $X$, and let $\mathcal{B} \varsubsetneq X$  be a countable union of proper subvarieties. Then either
\[\epsilon(L,1) \geq \gamma(L)\]  
or there exists a subvariety $Y \nsubseteq \mathcal{B}$ with $0 < \dim Y < d$ such that
\[\epsilon(L,1) \geq  \max\{\epsilon(L_{|Y},1), \gamma(L_{|Y})\}.\]
 \end{coro}

Note that
  \[\gamma(L) =  \max\left\{\frac{\sqrt[d]{(d+1)L^d}}{d+1}, \frac{\epsilon_1(L)}{d} \right\} \geq  \frac{\sqrt[d]{L^d}}{d},\]
  with strict inequality when $d \geq 2$.
 Therefore Corollary \ref{coromain} improves \cite[Theorem 3.1]{EKL}, which states that
 \[\epsilon(L,1) \geq \inf_{Y \nsubseteq \mathcal{B}} \frac{\sqrt[\dim Y]{L^{\dim Y} \cdot Y}}{\dim Y} \]
 where the infimum is over all subvarieties $Y \nsubseteq \mathcal{B}$ of dimension $\dim Y \geq 1$. 
 
\begin{proof} If $d =1$, then $\epsilon(L,1) = \gamma(L)$ so the statement holds.  We assume by induction that $d \geq 2$ and that the result is known for projective varieties $Z$ of dimension $\dim Z \leq d-1$. After possibly shrinking $X \setminus \mathcal{B}$, we assume that $X\setminus \mathcal{B}$ is contained in the smooth locus of $X$ and that $L_{|Z}$ is big for any subvariety $Z \nsubseteq \mathcal{B}$. Note that 
\[\max_{1 \leq i \leq d} \epsilon_i(L) - \epsilon_{i+1}(L) \geq \frac{1}{d} \sum_{i=1}^d \epsilon_i(L) - \epsilon_{i+1}(L) = \frac{\epsilon_1(L)}{d},\]
hence $\max_{1 \leq i \leq d} \epsilon_i(L) - \epsilon_{i+1}(L) \geq \gamma(L)$ by Lemma \ref{ineqmin}. Assume that $\gamma(L)  > \epsilon(L,1)$. In that case, we have  
\[\max_{1 \leq i \leq d-1} \epsilon_i(L) - \epsilon_{i+1}(L) \geq \gamma(L) > \alpha,\]
for some real number $\alpha > \epsilon(L,1)$. By Theorem \ref{thmmain} there exists a  subvariety $Z \nsubseteq \mathcal{B}$ with $0 < \dim Z < d$ such that
\[\epsilon(L,z) \geq \min \{\alpha, \epsilon(L_{|Z},1)\}\]
for a very general $z \in Z(k) \setminus \mathcal{B}$. Since $z$ is a smooth point of $X$ we have 
\[\epsilon(L,1) \geq \epsilon(L,z) \geq \min \{\alpha, \epsilon(L_{|Z},1)\}\]
 by Lemma \ref{lemmaverygeneral}, and therefore $\epsilon(L,1) \geq \epsilon(L_{|Z},1)$ since $\alpha > \epsilon(L,1)$. By the induction hypothesis, there exists a positive-dimensional subvariety $Y \subseteq Z \setminus \mathcal{B}$ such that 
 \[\epsilon(L,1) \geq \epsilon(L_{|Z},1) \geq \max \{\epsilon(L_{|Y},1), \gamma(L_{|Y})\}.\]
\end{proof}

 Finally, we note that Theorem \ref{thmmainintro} is a straightforward consequence of Corollary \ref{coromain}.
\begin{coro}\label{corobound} Let $L$ be a nef and big line bundle on a projective variety of dimension $d \geq 2$. 
Then 
\[\epsilon(L,1) \geq \min \{\gamma(L), d^{\frac{1}{d-1}-1}\} \geq (d+1)^{\frac{1}{d}-1}.\]
\end{coro}

\begin{proof} Let $\mathcal{B} \varsubsetneq X$ be a subvariety such that $L_{|Y}$ is big for any subvariety $Y\subseteq X$ not contained in $\mathcal{B}$. Then for any subvariety $Y \nsubseteq \mathcal{B}$ of dimension $r \in \{1, \ldots, d-1\}$, we have $L^r\cdot Y \geq 1$ and therefore 
\[\max\{\gamma(L_{|Y}), \epsilon(L_{|Y},1)\} \geq \gamma(L_{|Y}) \geq \frac{\sqrt[r]{(r+1)L^r\cdot Y}}{r+1} \geq (r+1)^{\frac{1}{r}-1}  \geq  d^{\frac{1}{d-1}-1}.\]
By Corollary \ref{coromain}, it follows that 
\[\epsilon(L,1) \geq \min \{\gamma(L), d^{\frac{1}{d-1}-1}\}\geq (d+1)^{\frac{1}{d}-1}.\]
\end{proof}

\bibliographystyle{alpha}
\bibliography{seshadritower.bib}

\begin{thebibliography}{dFEM14}

\bibitem[AI20]{AmbroIto}
Florin Ambro and Atsushi Ito.
\newblock Successive minima of line bundles.
\newblock {\em Adv. Math.}, 365:107045, 38, 2020.

\bibitem[CN14]{CasciniNakamaye}
Paolo Cascini and Michael Nakamaye.
\newblock Seshadri constants on smooth threefolds.
\newblock {\em Adv. Geom.}, 14(1):59--79, 2014.

\bibitem[Dem92]{Dem92}
Jean-Pierre Demailly.
\newblock Singular {H}ermitian metrics on positive line bundles.
\newblock In {\em Complex algebraic varieties ({B}ayreuth, 1990)}, volume 1507
  of {\em Lecture Notes in Math.}, pages 87--104. Springer, Berlin, 1992.

\bibitem[Den21]{Deng}
Ya~Deng.
\newblock Applications of the {O}hsawa-{T}akegoshi extension theorem to direct
  image problems.
\newblock {\em Int. Math. Res. Not. IMRN}, (23):17611--17633, 2021.

\bibitem[dFEM14]{DFEM}
Tommaso de~Fernex, Lawrence Ein, and Mircea Musta\c{t}\u{a}.
\newblock {\em Vanishing theorems and singularities in birational geometry}.
\newblock 2014.
\newblock \url{http://homepages.math.uic.edu/~ein/DFEM.pdf}.

\bibitem[EKL95]{EKL}
Lawrence Ein, Oliver K\"uchle, and Robert Lazarsfeld.
\newblock Local positivity of ample line bundles.
\newblock {\em J. Differential Geom.}, 42(2):193--219, 1995.

\bibitem[EL93]{EL}
Lawrence Ein and Robert Lazarsfeld.
\newblock Seshadri constants on smooth surfaces.
\newblock {\em Ast\'erisque}, (218):177--186, 1993.
\newblock Journ\'ees de G\'eom\'etrie Alg\'ebrique d'Orsay (Orsay, 1992).

\bibitem[Laz04]{Lazarsfeld}
Robert Lazarsfeld.
\newblock {\em Positivity in algebraic geometry. {I}}, volume~48 of {\em
  Ergebnisse der Mathematik und ihrer Grenzgebiete. 3. Folge. A Series of
  Modern Surveys in Mathematics [Results in Mathematics and Related Areas. 3rd
  Series. A Series of Modern Surveys in Mathematics]}.
\newblock Springer-Verlag, Berlin, 2004.
\newblock Classical setting: line bundles and linear series.

\bibitem[Nak05]{Na04}
Michael Nakamaye.
\newblock Seshadri constants at very general points.
\newblock {\em Trans. Amer. Math. Soc.}, 357(8):3285--3297, 2005.

\end{thebibliography}

\vspace*{0.5cm}

\vspace*{1cm}

\noindent François Ballaÿ\\
 Université Clermont Auvergne, CNRS, LMBP, F-63000 Clermont-Ferrand, France. \\
 \href{mailto:francois.ballay@uca.fr}{\texttt{francois.ballay@uca.fr}}\\
 \href{http://fballay.perso.math.cnrs.fr}{\texttt{fballay.perso.math.cnrs.fr}}

\end{document}